\documentclass[letterpaper,12pt,reqno]{amsart}
\usepackage{amscd}
\usepackage{indentfirst}
\usepackage{graphics}
\numberwithin{equation}{section}
\usepackage[margin=2.9cm]{geometry}
\usepackage{epstopdf}
\usepackage{adjustbox}
\usepackage{quiver}

\usepackage{amsmath, amsthm, amsfonts, amssymb}
\usepackage{relsize}
\usepackage{bbm}
\usepackage{mathtools}
\usepackage{stmaryrd}
\usepackage{graphicx}
\usepackage{tikz-cd}
\usepackage{enumitem}

\definecolor{cof}{RGB}{219,144,71}
\definecolor{pur}{RGB}{186,146,162}
\definecolor{greeo}{RGB}{91,173,69}
\definecolor{greet}{RGB}{52,111,72}

\usepackage{hyperref}
\hypersetup{
    colorlinks=true,
    linkcolor=violet,
    filecolor=magenta,      
    urlcolor=blue,
    citecolor=red
}
\usetikzlibrary{decorations.markings,positioning}
\definecolor{pur}{RGB}{186,146,162}
\pgfdeclarelayer{background}
\pgfsetlayers{background,main}
\newtheorem{thm}{Theorem}[section]
\newtheorem*{thm*}{Main Theorem}
\newtheorem{prop}[thm]{Proposition}
\newtheorem{lem}[thm]{Lemma}
\newtheorem{cor}[thm]{Corollary}

\theoremstyle{definition}
\newtheorem{defn}[thm]{Definition}
\newtheorem{ex}[thm]{Example}
\newtheorem{rmk}[thm]{Remark}
\newtheorem{para}[thm]{}

\newcommand{\tri}{\ol{\nabla}}

\newcommand{\bV}{\mathbb{V}}
\newcommand{\boxprod}{\hspace{0.2mm}\Box\hspace{0.2mm}}
\newcommand{\lifts}{\boxslash}

\newcommand{\cin}{\subseteq}

\newcommand\ol[1]{\ensuremath{\overline{#1}}}

\newcommand{\nospace}[1]{\makebox[0pt][l]{\,#1}}

\newcommand{\mcT}{\mathcal{T}}

\newcommand{\mcA}{\mathcal{A}}
\newcommand{\mcS}{\mathcal{S}}

\newcommand{\mcJ}{\mathcal{J}}

\newcommand{\mcB}{\mathcal{B}}
\newcommand{\mcN}{\mathcal{N}}
\newcommand{\mcM}{\mathcal{M}}

\newcommand{\op}{\operatorname{op}}
\newcommand{\id}{\operatorname{id}}
\newcommand{\sk}{\operatorname{sk}}

\newcommand{\IsoHorn}{\mathsf{IsoHorn}}
\newcommand{\Bdry}{\mathsf{Bdry}}
\newcommand{\Mono}{\mathsf{Mono}}
\newcommand{\Wide}{\mathsf{Wide}}

\newcommand{\Set}{\mathsf{Set}}
\newcommand{\sSet}{\mathsf{sSet}}

\newcommand{\Cat}{\mathsf{Cat}}

\begin{document}

\title[Fibrant objects in the minimal model structure on simplicial sets]{A horn-like characterization of the fibrant objects in the minimal model structure on simplicial sets}

\author[M. Feller]{Matt Feller}

\address{Department of Mathematics, University of Virginia, Charlottesville, VA 22904}

\email{feller@virginia.edu}

\date{\today}

\thanks{The author was partially supported by NSF RTG grant DMS-1839968 and NSF grant DMS-1906281.}

\maketitle

\begin{abstract}
    We show that the fibrant objects in the minimal model structure on the category of simplicial sets are characterized by a lifting condition with respect to maps which resemble the horn inclusions that define Kan complexes.
\end{abstract}

\setcounter{tocdepth}{1}

\tableofcontents

\section{Introduction}

Model categories play a crucial role in modern homotopy theory, underpinning much of the current work in higher categories. A model category consists of a category with a chosen model structure, which amounts to a choice of ``homotopy theory'' for the given category. More precisely, a model structure is a choice of weak equivalences, cofibrations, and fibrations satisfying certain axioms which abstract the behavior of the analogous classes of maps from the ordinary homotopy theory of topological spaces. (See \cite[Def.~2.2.1]{Cisinski:Cambridge} for an explicit definition of model categories.)

Often times, one can model different homotopy theories with a single category by constructing multiple model structures on that category. A major example is $\sSet$, the category of simplicial sets, which admits the Kan-Quillen model structure (modeling topological spaces) \cite{Quillen} as well as many others such as the Joyal model structure (modeling $(\infty,1)$-categories) \cite{Joyal:theory}. Chapter 6 of \cite{Balchin} contains an overview of most of the known model structures on $\sSet$. A common feature of these model structures is that they are examples of \emph{Cisinski model structures}, meaning that the cofibrations are the monomorphisms and the model structure is cofibrantly generated (the latter being a technical condition that we explain in Subsection \ref{sub:modelstructures}).

In general in a model category, we can consider an object $X$ well-behaved if it is \emph{cofibrant} (the map from the initial object $\varnothing\to X$ is a cofibration) as well as \emph{fibrant} (the map to the terminal object $X\to \ast$ is a fibration), but in a Cisinski model structure every object is cofibrant because $\varnothing\to X$ is a monomorphism, so the focus is on the class of fibrant objects. One can view the fibrant objects as those that ``actually behave like what we are modeling.'' For example, the fibrant objects in the Kan-Quillen model structure are the Kan complexes, which are precisely the simplicial sets which behave like topological spaces in a particular sense. From this perspective, understanding a Cisinski model structure on $\sSet$ amounts, in large part, to understanding the fibrant objects. In fact, it follows from a result of Joyal that a Cisinski model structure is uniquely determined by its class of fibrant objects; see Proposition \ref{prop:MSdeterminedbyfibobj}.

A common way to get a new model structure from an existing one is through \emph{localization}, where we fix the cofibrations and add more weak equivalences, yielding a new model structure with fewer fibrant objects. For example, if we start with the Joyal model structure and localize at some set of maps, then the fibrant objects of the resulting model structure are quasi-categories satisfying some extra property, such as the $n$-truncated quasi-categories in \cite{CL}. However, if we wish to find a model structure with a class of fibrant objects generalizing that of quasi-categories, as is our goal in a related paper \cite{Feller:generalizing}, there is no standard process for going the other direction and ``de-localizing'' the Joyal model structure. Luckily, in \cite{Cisinski:Asterisque} Cisinski provides a powerful framework for producing model structures on presheaf categories, such as $\sSet$. See \cite[Sec.~2.4]{Cisinski:Cambridge} for an English summary of his theory.

One of the consequences of Cisinski's theory is the existence of a \emph{minimal model structure} on $\sSet$, of which every other Cisinski model structure on $\sSet$ is a localization. From his machinery one also gets a description of the fibrant objects in terms of lifts of certain pushout-product maps, as we recall in Section \ref{sec:knownchar}. The goal of this paper is to provide a new characterization of the fibrant objects which is easier to check and more closely resembles the horn-lifting conditions that define Kan complexes and quasi-categories.

To illustrate the idea behind our new characterization, recall that a \emph{face} of an ordinary $n$-simplex is the result of deleting one of the vertices, and an ordinary $n$-horn is the union of all but one face of the $n$-simplex. For example, here is a picture of a 2-horn inclusion:
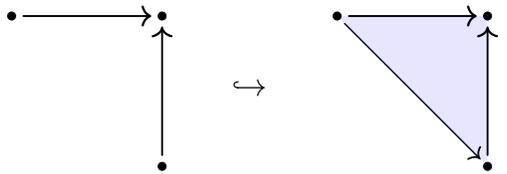
\[
\adjustbox{scale=1}{
\begin{tikzpicture}[line join = round, line cap = round]
\begin{scope}[decoration={markings,mark=at position .93 with {\arrow[scale=1.2]{>}}}]
\draw[line width=0.25mm,postaction={decorate}] (0,0) -- node[auto] {} (2,0);
\end{scope}
\filldraw[white] (0,0) circle (4pt);
\filldraw[white] (2,0) circle (4pt);
\filldraw[black] (0,0) circle (1.5pt);
\filldraw[black] (2,0) circle (1.5pt);
\filldraw[white] (2,-2) circle (4pt);
\filldraw[black] (2,-2) circle (1.5pt);
\begin{scope}[decoration={markings,mark=at position 1 with {\arrow[scale=1.2]{>}}}]
\draw[line width=0.25mm,postaction={decorate}] (2,-1.85) -- node[auto] {} (2,-0.15);
\end{scope}
\end{tikzpicture}
\hspace{.2in}
$\mathrel{\raisebox{6ex}{$\hookrightarrow$}}$
\hspace{.2in}
\begin{tikzpicture}[line join = round, line cap = round]
\begin{scope}[decoration={markings,mark=at position .93 with {\arrow[scale=1.2]{>}}}]
\draw[line width=0.25mm,postaction={decorate}] (0,0) -- node[auto] {} (2,0);
\end{scope}
\filldraw[white] (0,0) circle (4pt);
\filldraw[white] (2,0) circle (4pt);
\filldraw[black] (0,0) circle (1.5pt);
\filldraw[black] (2,0) circle (1.5pt);
\filldraw[white] (2,-2) circle (4pt);
\filldraw[black] (2,-2) circle (1.5pt);
\path[fill=blue,fill opacity=.1] (0,0)--(2,0)--(2,-2)--cycle;
\begin{scope}[decoration={markings,mark=at position 1 with {\arrow[scale=1.2]{>}}}]
\draw[line width=0.25mm,postaction={decorate}] (2,-1.85) -- node[auto] {} (2,-0.15);
\draw[line width=0.2mm,postaction={decorate}] (.1,-.1) -- node[auto] {} (1.9,-1.9);
\end{scope}
\end{tikzpicture}
}.
\]
Also recall that the 2-simplex is the nerve of the category generated by two arrows
\[
c_0\to c_1 \to c_2 .
\]
For our characterization, we define \emph{isoplexes}, which turn out to be the nerve of a category generated by $n$ arrows, except that one of the arrows $c_i\to c_{i+1}$ is an isomorphism. We can define faces of isoplexes to be the result of deleting vertices, and \emph{iso-horns} to be the union of all but a certain face of an isoplex. For example, here is a picture of an iso-horn inclusion:
\[
\adjustbox{scale=1}{
\begin{tikzpicture}[line join = round, line cap = round]
\begin{scope}[decoration={markings,mark=at position .93 with {\arrow[scale=1.2]{>}}}]
\draw[line width=0.25mm,red,postaction={decorate}] (0,0) -- node[auto] {} (2,0);
\end{scope}
\filldraw[white] (0,0) circle (4pt);
\filldraw[white] (2,0) circle (4pt);
\filldraw[red] (0,0) circle (1.5pt);
\filldraw[magenta] (2,0) circle (1.5pt);
\filldraw[white] (2,-2) circle (4pt);
\filldraw[blue] (2,-2) circle (1.5pt);
\begin{scope}[decoration={markings,mark=at position 1 with {\arrow[scale=1.2]{>}}}]
\draw[line width=0.25mm,blue,postaction={decorate}] (1.9,-0.15) -- node[auto] {} (1.9,-1.85);
\draw[line width=0.25mm,blue,postaction={decorate}] (2.1,-1.85) -- node[auto] {} (2.1,-0.15);
\end{scope}
\end{tikzpicture}
\hspace{.2in}
$\mathrel{\raisebox{6ex}{$\hookrightarrow$}}$
\hspace{.2in}
\begin{tikzpicture}[line join = round, line cap = round]
\begin{scope}[decoration={markings,mark=at position .93 with {\arrow[scale=1.2]{>}}}]
\draw[line width=0.25mm,postaction={decorate}] (0,0) -- node[auto] {} (2,0);
\end{scope}
\filldraw[white] (0,0) circle (4pt);
\filldraw[white] (2,0) circle (4pt);
\filldraw[black] (0,0) circle (1.5pt);
\filldraw[black] (2,0) circle (1.5pt);
\filldraw[white] (2,-2) circle (4pt);
\filldraw[black] (2,-2) circle (1.5pt);
\path[fill=magenta,fill opacity=.1] (0,0)--(2,0)--(2.1,-0.1)--(2.1,-1.9)--(2,-2)--cycle;
\begin{scope}[decoration={markings,mark=at position 1 with {\arrow[scale=1.2]{>}}}]
\draw[line width=0.25mm,postaction={decorate}] (1.9,-0.15) -- node[auto] {} (1.9,-1.85);
\draw[line width=0.25mm,postaction={decorate}] (2.1,-1.85) -- node[auto] {} (2.1,-0.15);
\draw[line width=0.2mm,postaction={decorate}] (.1,-.1) -- node[auto] {} (1.75,-1.75);
\end{scope}
\end{tikzpicture}
}.
\]

\begin{thm*}[Thm.~\ref{thm:isohornminimalfibrant}]
A simplicial set $X$ is fibrant in the minimal model structure if and only if it has lifts with respect to all iso-horn inclusions.
\end{thm*}

One direction of this result follows from the observation that these iso-horn inclusions are retracts of the pushout-product maps from the known characterization of the fibrant objects in the minimal model structure, such as
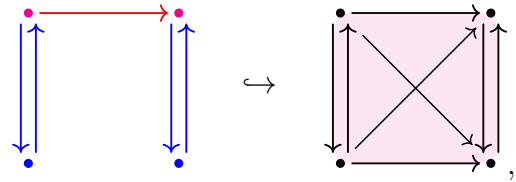
\[
\begin{tikzpicture}[line join = round, line cap = round]
\begin{scope}[decoration={markings,mark=at position .93 with {\arrow[scale=1]{>}}}]
\draw[line width=0.25mm,red,postaction={decorate}] (0,0) -- node[auto] {} (2,0);
\end{scope}
\filldraw[white] (0,0) circle (4pt);
\filldraw[white] (2,0) circle (4pt);
\filldraw[magenta] (0,0) circle (1.5pt);
\filldraw[magenta] (2,0) circle (1.5pt);
\filldraw[white] (0,-2) circle (4pt);
\filldraw[white] (2,-2) circle (4pt);
\filldraw[blue] (0,-2) circle (1.5pt);
\filldraw[blue] (2,-2) circle (1.5pt);
\begin{scope}[decoration={markings,mark=at position 1 with {\arrow[scale=1]{>}}}]
\draw[line width=0.25mm,blue,postaction={decorate}] (-0.1,-0.15) -- node[auto] {} (-0.1,-1.85);
\draw[line width=0.25mm,blue,postaction={decorate}] (1.9,-0.15) -- node[auto] {} (1.9,-1.85);
\draw[line width=0.25mm,blue,postaction={decorate}] (0.1,-1.85) -- node[auto] {} (0.1,-0.15);
\draw[line width=0.25mm,blue,postaction={decorate}] (2.1,-1.85) -- node[auto] {} (2.1,-0.15);
\end{scope}
\end{tikzpicture}
\hspace{.2in}
\mathrel{\raisebox{6ex}{$\hookrightarrow$}}
\hspace{.2in}
\begin{tikzpicture}[line join = round, line cap = round]
\begin{scope}[decoration={markings,mark=at position .93 with {\arrow[scale=1]{>}}}]
\draw[line width=0.25mm,postaction={decorate}] (0,0) -- node[auto] {} (2,0);
\draw[line width=0.25mm,postaction={decorate}] (0,-2) -- node[auto] {} (2,-2);
\end{scope}
\filldraw[white] (0,0) circle (4pt);
\filldraw[white] (2,0) circle (4pt);
\filldraw[black] (0,0) circle (1.5pt);
\filldraw[black] (2,0) circle (1.5pt);
\filldraw[white] (0,-2) circle (4pt);
\filldraw[white] (2,-2) circle (4pt);
\filldraw[black] (0,-2) circle (1.5pt);
\filldraw[black] (2,-2) circle (1.5pt);
\path[fill=magenta,fill opacity=.1] (-0.1,-0.2)--(0,0)--(2,0)--(2.1,-0.2)--(2.1,-1.8)--(2,-2)--(0,-2)--(-0.1,-1.8)--cycle;
\begin{scope}[decoration={markings,mark=at position 1 with {\arrow[scale=1]{>}}}]
\draw[line width=0.25mm,postaction={decorate}] (-0.1,-0.15) -- node[auto] {} (-0.1,-1.85);
\draw[line width=0.25mm,postaction={decorate}] (1.9,-0.15) -- node[auto] {} (1.9,-1.85);
\draw[line width=0.25mm,postaction={decorate}] (0.1,-1.85) -- node[auto] {} (0.1,-0.15);
\draw[line width=0.25mm,postaction={decorate}] (2.1,-1.85) -- node[auto] {} (2.1,-0.15);
\draw[line width=0.2mm,postaction={decorate}] (.3,-.3) -- node[auto] {} (1.75,-1.75);
\draw[line width=0.2mm,postaction={decorate}] (.2,-1.8) -- node[auto] {} (1.8,-.2);
\end{scope}
\end{tikzpicture},
\]
and hence that the fibrant objects in the minimal model structure must have lifts with respect to iso-horn inclusions. What is not as straightforward is to show that having lifts with respect to iso-horn inclusions is sufficient to imply fibrancy. In order to prove the latter direction, we introduce the concepts of widenings and widened inclusions in Section \ref{sec:wide}, which also ultimately provide us a more conceptual understanding of fibrancy in the minimal model structure.

More specifically, what we show in Section \ref{sec:isohornchar} is that the class of pushout-products of inclusions with the map ${0}\hookrightarrow J$ generate the same saturated class as the set of iso-horn inclusions. This result turns out to be very useful in the proof of one of the main results of \cite{Feller:quasi}, that the appropriate definition of completeness for 2-Segal spaces gives us Quillen equivalences of model structures with our model structure for quasi-2-Segal sets. The key piece of the argument is to show something about the saturated class generated by the pushout-products of inclusions with the map ${0}\hookrightarrow J$; by applying the results from this paper, we reduce the argument to a straightforward lemma about iso-horns.

With a better handle on the minimal model structure comes a deeper understanding more generally of Cisinski model structures on $\sSet$. For one thing, it allows us to conclude that a class of simplicial sets cannot be the class of fibrant objects for some Cisinski model structure if we know that it does not have lifts with respect to iso-horn inclusions. It also shows us what aspects are universal to all Cisinski model structures on simplicial sets, as opposed to what aspects might be unique to a particular model structure such as the Joyal or Kan-Quillen model structures. This idea is illustrated in \cite{Feller:generalizing}, where we identify a nice property of the Joyal model structure which we call ``homotopically-behaved.'' The minimal model structure itself does not turn out to be homotopically-behaved, meaning that there is some subtlety to be addressed when constructing model structures with fibrant objects more general than quasi-categories if we would like the model structure to be homotopically-behaved. This subtlety motivates our central theorem in \cite{Feller:generalizing}, which is that there exists a minimal homotopically-behaved model structure on $\sSet$.

Although our methods as laid out in the present paper only apply to the category of simplicial sets, they may be useful as a blueprint for the study of the minimal model structure on other presheaf categories. We briefly speculate about how one might approach this generalization in Remark \ref{rmk:othercategories}.

\subsection{Outline}

Section \ref{sec:background} covers the necessary background in simplicial sets, model categories, and Cisinski's theory, followed by Section \ref{sec:knownchar}, in which we recall the known characterization of the fibrant objects in the minimal model structure. In Section \ref{sec:wide}, we introduce widenings and widened inclusions and prove some fundamental properties, and then in Section \ref{sec:isohornchar} we define iso-horns and show that the fibrant objects in the minimal model structure are characterized by having lifts with respect to iso-horn inclusions.

\subsection{Acknowledgements}

I would like to thank Julie Bergner and Scott Balchin for their helpful and detailed feedback on early drafts.

\section{Background}\label{sec:background}

In this section we set our terminology and notation relating to simplicial sets and model categories, and then briefly review Cisinski's theory.

\subsection{Simplicial sets}\label{sub:simplicialsets}

The \emph{simplex category $\Delta$} is the category with objects
\[
[n]=\{0\leq 1\leq \ldots\leq n\}\text{ for each }n\geq 0,
\]
whose morphisms are monotone maps. A \emph{simplicial set} is a functor $\Delta^{\op}\to \Set$. We denote the category of simplicial sets by $\sSet$. We denote by $\Delta[n]$ the standard $n$-simplex, which is the simplicial set corresponding to $[n]$ via the Yoneda embedding $\Delta\hookrightarrow \sSet$. See Sections 1.1 and 1.2 of \cite{Cisinski:Cambridge} for more details about $\sSet$.

Given a simplicial set $X\colon \Delta^{\op}\to \Set$, we denote by $X_n$ the set $X([n])$, and call an element in $X_n$ an \emph{$n$-simplex} of $X$. We often refer to a 0-simplex as a \emph{vertex} and a 1-simplex as an \emph{edge}. By the Yoneda Lemma, an $n$-simplex in $X$ is equivalently a map of simplicial sets $\Delta[n]\to X$.

The injective morphisms in $\Delta$ are generated by co-face maps $d^i\colon [n-1]\hookrightarrow [n]$, which give us the \emph{face maps} $d_i\colon X_n\to X_{n-1}$ of a simplicial set $X$ for each $n\geq 1$ and $0\leq i\leq n$. The \emph{boundary} of the standard $n$-simplex $\Delta$ is the union of all of its faces, denoted by $\partial\Delta[n]$. The \emph{boundary of an $n$-simplex} $\sigma\colon \Delta[n]\to X$ is its restriction along the \emph{boundary inclusion} $\partial\Delta[n]\hookrightarrow \Delta[n]$.

Given a simplicial set $X$ and a set of vertices $\nu\cin X$, the \emph{full subcomplex of $X$ on $\nu$} is the simplicial set $Z\cin X$ consisting of the simplices of $X$ whose vertices are all in $\nu$.

The surjective morphisms in $\Delta$ are generated by co-degeneracy maps $s^i\colon [n+1]\to [n]$, which give us the \emph{degeneracy maps} $s_i\colon X_n\to X_{n+1}$ of a simplicial set $X$ for each $n\geq 0$ and $0\leq i\leq n$. An $n$-simplex is \emph{degenerate} if it is in the image of a degeneracy map, and is \emph{non-degenerate} otherwise. Given a simplicial set $X$ and $k\geq 0$, the \emph{$k$-skeleton of $X$} is the simplicial set $\sk_k X\cin X$ generated by the non-degenerate $n$-simplices of $X$ for all $n\leq k$.

\subsection{The nerve functor and the simplicial set $J$}\label{sub:nerve}

Denote by $\Cat$ the category of small categories. The \emph{nerve functor} is a fully faithful embedding $N\colon \Cat\hookrightarrow \sSet$; see Section 1.4 of \cite{Cisinski:Cambridge}. Denote by $\mathbb{I}$ the free-living isomorphism, i.e., the category with two objects and exactly one morphism in each hom-set. We let $J=N(\mathbb{I})$, and denote its two vertices by 0 and 1. We let $\partial J=\{0\}\cup\{1\}$ in $J$. Since each simplex of $J$ is uniquely determined by its vertices, throughout this paper we denote an $n$-simplex of $J$ by an element $(a_0,\ldots, a_n)$ of $\{0,1\}^{n+1}$.

\subsection{Lifting properties and saturated classes}\label{sub:lifts}

We say that a map $g\colon X\to Y$ has the \emph{right lifting property with respect to} $f\colon A\to B$, denoted by $f\lifts g$, if for all commutative squares
\[
\begin{tikzcd}
{A} \arrow[d, "f"'] \arrow[r, "u"]                      & {X} \arrow[d, "g"] \\
{B} \arrow[r, "v"'] \arrow[ru, "\exists \ell"', dotted] & {Y} \nospace{,}             
\end{tikzcd}
\]
there exists a \emph{lift}, i.e., a dotted arrow $\ell$ making each triangle commute. More generally, we say that a map $g$ has the \emph{right lifting property with respect to} a class of maps $\mathcal{B}$, denoted by $\mathcal{B}\lifts g$, if $f\lifts g$ for all $f$ in $\mathcal{B}$. The class of morphisms with the right lifting property with respect to $\mathcal{B}$ is denoted by $\mathcal{B}^{\lifts}$.

A class of morphisms is \emph{saturated} if it is closed under taking isomorphisms, pushouts, transfinite compositions, and retracts. Given a class of morphisms $\mcB$, we can take its \emph{saturated closure} $\ol{\mcB}$. If a class $\mcB'$ equals $\ol{\mcB}$, we say that $\mcB'$ is \emph{generated by $\mcB$}. For example, the set of \emph{boundary inclusions} $\Bdry=\{\partial\Delta[n]\to\Delta[n]\}_{n\geq 0}$ generates the class $\Mono$ of monomorphisms in $\sSet$. For any class $\mcB$, we have $\mcB{}^{\lifts}=\ol{\mcB}{}^{\ \lifts}$, and for any classes of maps $\mcS$ and $\mcT$, the containment $\mcS\cin \ol{\mcT}$ implies $\ol{\mcS}\cin \ol{\mcT}$. See Section 2.1 of \cite{Cisinski:Cambridge} for more details about lifting properties and saturated classes.

\subsection{Pushout-products}\label{sub:pushoutproducts} Although one can define pushout-products of arbitrary morphisms in a monoidal category, for our purposes it suffices to consider monomorphisms in the monoidal category $(\sSet,\times)$. Given monomorphisms $A\hookrightarrow B$ and $W\hookrightarrow Z$ of simplicial sets, the monomorphism
\[
(Z\times A)\cup(W\times B)\hookrightarrow Z\times B
\]
is called the \emph{pushout-product} of $A\hookrightarrow B$ and $W\hookrightarrow Z$, denoted by $(A\hookrightarrow B)\boxprod (W\hookrightarrow Z)$. Given two classes $\mcS$ and $\mcT$ of maps, we denote by $\mcS\boxprod\mcT$ the class of maps of the form $f\boxprod g$ for $f$ in $\mcS$ and $g$ in $\mcT$.

\subsection{Model structures}\label{sub:modelstructures}

A \emph{model structure} on a complete and cocomplete category consists of a choice of three classes of morphisms, the \emph{cofibrations}, the \emph{fibrations}, and the \emph{weak equivalences}, subject to certain axioms; see \cite[Def.~2.2.1]{Cisinski:Cambridge} or \cite{Hovey}. We say that a morphism which is both a (co)fibration and a weak equivalence is a \emph{trivial (co)fibration}. The \emph{fibrant objects} are those such that the map to the terminal object is a fibration, and the \emph{cofibrant objects} are those such that the map from the initial object is a cofibration.

Instead of describing the axioms of a model category in general, we restrict our focus to \emph{Cisinski model structures} on $\sSet$, which are cofibrantly generated model structures whose cofibrations are precisely the monomorphisms. A model structure is \emph{cofibrantly generated} if there are sets $\mathcal{I}$ and $\mathcal{J}$ such that $\mathcal{I}$ generates the cofibrations and $\mathcal{J}$ generates the trivial cofibrations in the sense of Subsection \ref{sub:lifts}; see \cite[2.4.1]{Cisinski:Cambridge}. Given two Cisinski model structures $\mathcal{M}$ and $\mathcal{M}'$ on $\sSet$ whose classes of weak equivalences are $W$ and $W'$ respectively, we say that $\mathcal{M}'$ is a \emph{localization of $\mathcal{M}$} if $W'\supseteq W$.

It is a standard fact that every Cisinski model structure is left proper because all objects are cofibrant; e.g., see Proposition 13.1.2 in \cite{Hirschhorn}. We state this fact explicitly in the following lemma for later reference.

\begin{lem}\label{lem:leftproper}
In a Cisinski model structure, pushouts along inclusions preserve weak equivalences.
\end{lem}

Some properties of arbitrary model structures are also relevant for our purposes:
\begin{enumerate}
    \item The class of weak equivalences satisfies the 2-out-of-3 property: if two of $f$, $g$, and $gf$ are weak equivalence, then so is the third.
    \item The trivial fibrations are precisely the maps with the right lifting property with respect to all cofibrations.
    \item The class of trivial cofibrations is closed under pushouts, transfinite compositions, and retracts.
\end{enumerate}

Another item that could be added to the list above is that the fibrations are precisely the maps with the right lifting property with respect to the trivial cofibrations. In fact, if the trivial cofibrations are generated by a set $\mathcal{J}$, then a map $f$ is a fibration as long as $\mathcal{J}\lifts f$. In particular, an object $X$ is fibrant if and only if $\mathcal{J}\lifts(X\to \ast)$. However, it is often the case that one can only prove abstractly that such a set $\mcJ$ exists. A major example is the Joyal model structure, where it is still an open problem to provide an explicit set of maps which generates the trivial cofibrations. However, in the Joyal model structure we have a particular set of trivial cofibrations, called the inner horns, which identifies the fibrant objects. Cisinski's machinery, which we review in Subsection \ref{sub:Cisinski}, always yields such a set. In Section \ref{sec:knownchar}, we explain how the set $(\{0\}\hookrightarrow J)\boxprod \Bdry$ identifies the fibrant objects in the minimal model structure, and argue that the minimal model structure is truly minimal, relying on the following proposition of Joyal.

\begin{prop}\label{prop:MSdeterminedbyfibobj}\cite[Prop.~E.1.10]{Joyal:theory}
A model structure is determined by its class of cofibrations together with its class of fibrant objects. Furthermore, if two model structures $\mcM=(C,W,F)$ and $\mcM'=(C,W',F')$ on a category have the same cofibrations, then $W\cin W'$ if and only if every fibrant object in $\mcM'$ is also fibrant in $\mcM$.
\end{prop}

\subsection{Cisinski's theory}\label{sub:Cisinski}

The following definition is essential to Cisinski's theory.

\begin{defn}\cite[Def.~2.4.11]{Cisinski:Cambridge}\label{def:anodyneclass}
We say that a class of morphisms $\ol{\Lambda}$ generated by a set $\Lambda$ of monomorphisms is a \emph{$(J\times -)$-anodyne class} if the following conditions hold.
\begin{enumerate}[start=1,label={(An\arabic*).\ }, widest=(An2$'$.), leftmargin=*]
\item The class $(\{0\}\hookrightarrow J)\boxprod\Mono$ is contained in $\ol{\Lambda}$.
\item For each $A\hookrightarrow B$ in $\ol{\Lambda}$, the pushout-product $(\partial J\hookrightarrow J)\boxprod (A\hookrightarrow B)$ is also in $\ol{\Lambda}$.
\end{enumerate}
\end{defn}

We can restate each of axioms (An1) and (An2) in a form which is easier to check.

\begin{lem}\cite[Lem.~2.10]{Feller:generalizing}\label{lem:altaxioms}
Let $\Lambda$ be a set of monomorphisms. Then axiom \emph{(An1)} is equivalent to \emph{(An1$'$)} below and axiom \emph{(An2)} is equivalent to \emph{(An2$'$)} below.
\end{lem}
\begin{enumerate}[start=1,label={(An\arabic*$'$).\ }, widest=(An2$'$.), leftmargin=*]
\item The class $(\{0\}\hookrightarrow J)\boxprod\Bdry$ is contained in $\ol{\Lambda}$.
\item For each $A\hookrightarrow B$ in $\Lambda$, the pushout-product $(\partial J\hookrightarrow J)\boxprod (A\hookrightarrow B)$ is in $\ol{\Lambda}$.
\end{enumerate}

We refer to \cite{Feller:generalizing} for a proof of Lemma \ref{lem:altaxioms}, although we note that proving the equivalence of (An1) and (An1$'$) amounts to proving the following fact which we state as a lemma for later reference.

\begin{lem}\label{lem:equivofan1givespushoutprod}
The following classes are equal:
\[
\ol{(\{0\}\hookrightarrow J)\boxprod \Bdry}=\ol{(\{0\}\hookrightarrow J)\boxprod \Mono}.
\]
\end{lem}

The main result we need from Cisinski is that for every $(J\times -)$-anodyne class $\ol{\Lambda}$, there is a Cisinski model structure such that $\Lambda$ identifies the fibrant objects in that model structure.

\begin{thm}\label{thm:cisinskiMS}\cite[Thm.~2.4.19]{Cisinski:Cambridge}
Given a set of monomorphisms $\Lambda$ such that $\ol{\Lambda}$ is a $(J\times-)$-anodyne class, there is a cofibrantly generated model structure on $\sSet$ whose cofibrations are the monomorphisms and whose fibrant objects are the simplicial sets with the right lifting property with respect to $\Lambda$.
\end{thm}

\begin{rmk}
Cisinski's results are much more general than stated here, but this formulation suffices for our purposes in this paper.
\end{rmk}

\section{The usual characterization of fibrant objects in the minimal model structure}\label{sec:knownchar}

In this section we recall the known characterization of the minimal model structure. First, recall that the map $J\to \ast$ is a trivial fibration (and hence a weak equivalence) in any Cisinski model structure because it has the right lifting property with respect to all monomorphisms. Furthermore, taking the pullback along $X\to \ast$ shows that for any simplicial set $X$ the projection map $J\times X\to X$ is also a weak equivalence. By the 2-out-of-3 property, the inclusion $\{0\}\times \id\colon X\hookrightarrow J\times X$ is therefore a weak equivalence too. In particular, for all $n\geq 0$ the left vertical map and the rightmost curved map in the diagram
\[
\begin{tikzcd}
	{\{0\}\times\partial\Delta[n]} & {\{0\}\times\Delta[n]} \\
	{J\times\partial\Delta[n]} & {(J\times\partial\Delta[n])\cup(\{0\}\times\Delta[n])} \\
	&& {J\times\Delta[n]}
	\arrow[hook, from=1-1, to=1-2]
	\arrow["\sim"', hook, from=1-1, to=2-1]
	\arrow["\sim"', hook, from=1-2, to=2-2]
	\arrow[hook, from=2-1, to=2-2]
	\arrow[dotted, hook, from=2-2, to=3-3]
	\arrow["\sim", curve={height=-18pt}, hook, from=1-2, to=3-3]
	\arrow[curve={height=12pt}, hook, from=2-1, to=3-3]
\end{tikzcd}
\]
are weak equivalences. The vertical map on the right is also a weak equivalence by Lemma \ref{lem:leftproper} since the inner square is a pushout, so by the 2-out-of-3 property the induced map
\[\left(\{0\}\hookrightarrow J\right)\boxprod \left(\partial\Delta[n]\hookrightarrow \Delta[n]\right)\]
is a weak equivalence. Since each of these pushout-product maps are monomorphisms, we have proved the following proposition.

\begin{prop}\label{prop:pushoutprodWEs}
Let $\mcA=\left(\{0\}\hookrightarrow J\right)\boxprod \Bdry$. The maps in $\mcA$ are trivial cofibrations in every Cisinski model structure.
\end{prop}

As a consequence, if a simplicial set $X$ is fibrant in a Cisinski model structure, then $\mcA \lifts (X\to \ast)$. Therefore, if we can show that there is a Cisinski model structure whose fibrant objects are precisely the simplicial sets $X$ such that $\mcA \lifts (X\to \ast)$, such a model structure necessarily has the broadest possible class of fibrant objects and therefore the smallest class of weak equivalences by Proposition \ref{prop:MSdeterminedbyfibobj}.

It now remains to show that there is a model structure whose fibrant objects are the simplicial sets with lifts with respect to maps in $\mcA$. We can see right away that $\mcA$ satisfies axiom (An1$'$) of Lemma \ref{lem:altaxioms} (and therefore (An1) too) since the maps in $\mcA$ are precisely the maps demanded by axiom (An1$'$). Thus, it remains only to show that axiom (An2$'$) is satisfied. The crucial observation is that the maps in (An2$'$) are themselves the iterated pushout products
\[
(\partial J\hookrightarrow J)\boxprod \left((\{0\}\hookrightarrow J)\boxprod (\partial\Delta[n]\hookrightarrow \Delta[n])\right),
\]
but since the pushout-product inherits the commutativity and associativity of the Cartesian product, they can be rewritten as
\[
(\{0\}\hookrightarrow J)\boxprod \left((\partial J\hookrightarrow J)\boxprod (\partial\Delta[n]\hookrightarrow \Delta[n])\right).
\]
Since
\[
(\partial J\hookrightarrow J)\boxprod (\partial\Delta[n]\hookrightarrow \Delta[n])
\]
is a monomorphism, we know that its pushout-product with $\{0\}\hookrightarrow J$ is in $\ol{\mcA}$ because $\mcA$ satisfies axiom (An1). Thus, we see that axiom (An2$'$) is satisfied, proving the following proposition.

\begin{prop}\label{prop:Aanodyne}
The set $\mcA=\left(\{0\}\hookrightarrow J\right)\boxprod \Bdry$ generates a $(J\times -)$-anodyne class.
\end{prop}

\begin{cor}\label{cor:minimalMSexists}
There is a Cisinski model structure on simplicial sets where a simplicial set 
$X$ is fibrant if and only if $X\to \ast$ has the right lifting property with respect to the set $\mcA$. Every Cisinski model structure on simplicial sets is a localization of this model structure, which we call the \emph{minimal Cisinski model structure}.
\end{cor}

\begin{proof}
For the first claim, apply Theorem \ref{thm:cisinskiMS} to Proposition \ref{prop:Aanodyne}. For the second claim, by Proposition \ref{prop:pushoutprodWEs} we know that for every Cisinski model structure, the fibrant objects must have lifts with respect to maps in $\mcA$.
\end{proof}

\begin{rmk}
Cisinski points out the existence of a minimal model structure on any presheaf category in \cite{Cisinski:Asterisque}, using the subobject classifier of $\sSet$ instead of the simplicial set $J$. (For a discussion of the subobject classifier in a presheaf category such as $\sSet$, see Example 2.1.11 in \cite{Cisinski:Cambridge}.) However, for the purposes of describing the minimal model structure, the only relevant properties of the subobject classifier $\Omega$ is that it has two distinct 0-simplices and that $\Omega\to \ast$ has the right lifting property with respect to all monomorphisms, which are properties that $J$ has as well.
\end{rmk}

\begin{rmk}
The precise description in Corollary \ref{cor:minimalMSexists} in terms of the set $\mcA$ is a special case of the description given in Paragraph 2.10 of \cite{Ara} of a minimal anodyne class containing a given set $S$, where we let $S=\varnothing$.
\end{rmk}

\begin{para}\label{par:htpyextension}
We conclude this section by noting that we now have two equivalent conditions for fibrancy in the minimal model structure. On the one hand, we know that $X$ is fibrant if and only if $X\to \ast$ has the right lifting property with respect to $\mcA=\left(\{0\}\hookrightarrow J\right)\boxprod \Bdry$. Since $\mcA$ is countable and relatively easy to describe, this characterization is useful for trying to verify that a simplicial set is indeed fibrant. On the other hand, by Lemma \ref{lem:equivofan1givespushoutprod}, a fibrant simplicial set $X$ must have lifts with respect to all maps in $\left(\{0\}\hookrightarrow J\right)\boxprod \Mono$. We can view this latter condition as a homotopy extension property, which gives a useful conceptual understanding of the fibrant objects in the minimal model structure.
\end{para}

\section{Widenings and widened inclusions}\label{sec:wide}

In this section, we introduce the concepts of \emph{widenings} and \emph{widened inclusions}. To explain the idea behind widened inclusions, let us consider the ``homotopy extension property'' perspective mentioned in Paragraph \ref{par:htpyextension}. With this perspective, we view an extension along $(\{0\}\hookrightarrow J)\boxprod (X\hookrightarrow Y)$ as extending a map from $Y$ along a homotopy of the sub-complex $X$. The idea of a widened inclusion is exactly the same, except that this homotopy is fixed on a certain sub-complex of $Y$. Thus, the class of widened inclusions we define will contain the maps in $(\{0\}\hookrightarrow J)\boxprod \Mono$. We first make a retract argument to show that the fibrant objects in the minimal model structure have lifts with respect to widened inclusions, and so satisfy a slightly stronger version of the homotopy extension property. Then we prove a proposition about widened inclusions which will be important in Section \ref{sec:isohornchar}.

Let us begin by defining widenings.

\begin{defn}
Given a simplicial set $X$ and a set of vertices $\nu\cin X_0$, we define the \emph{widening of $X$ at $\nu$} to be the full subcomplex of $J\times X$ on the set $\left(\{0\}\times X_0\right)\cup \left(\{1\}\times \nu\right)$. We denote the widening of $X$ at $\nu$ by $W_{\nu}(X)$, although if $\nu$ has just one element $x$ then we write $W_x(X)$ instead of $W_{\{x\}}(X)$.
\end{defn}

\begin{ex}
Given a simplicial set $X$, the widening of $X$ at its set of vertices is $W_{X_0}(X)=J\times X$. The widening of $X$ at the empty set is $W_{\varnothing}(X)=\{0\}\times X$.
\end{ex}

\begin{ex}\label{ex:isoplexiswidening}
The simplicial set $W_2(\Delta[2])$ (i.e., the widening of $\Delta[2]$ at the vertex $2$) can be pictured as
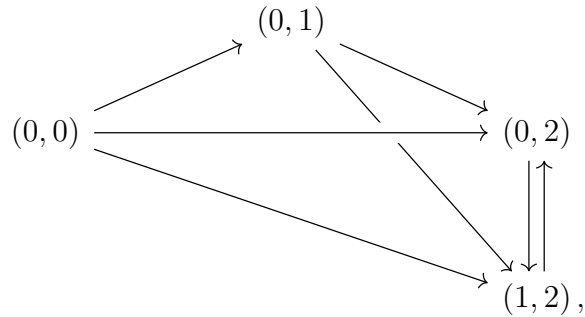
\[\begin{tikzcd}
	&& {(0,1)} \\
	{(0,0)} &&&& {(0,2)} \\
	\\
	&&&& {(1,2)\nospace{,}}
	\arrow[from=2-1, to=1-3]
	\arrow[from=1-3, to=2-5]
	\arrow[from=1-3, to=4-5]
	\arrow[shift right=1, from=2-5, to=4-5]
	\arrow[shift right=1, from=4-5, to=2-5]
	\arrow[from=2-1, to=2-5, crossing over]
	\arrow[from=2-1, to=4-5]
\end{tikzcd}\]
which is the nerve of a category that we can depict as
\[\begin{tikzcd}
	\bullet & \bullet & \bullet & \bullet\nospace{.}
	\arrow[from=1-1, to=1-2]
	\arrow[from=1-2, to=1-3]
	\arrow[shift right=1, from=1-3, to=1-4]
	\arrow[shift right=1, from=1-4, to=1-3]
\end{tikzcd}\]
This example is in fact what we call an \emph{isoplex}, which we define in Section \ref{sec:isohornchar} to be any widening of a standard $n$-simplex $\Delta[n]$ at a single vertex.
\end{ex}

\begin{rmk}
Given a small category $C$ and some set of objects $\nu$, the widening $W_\nu\left(N(C)\right)$ is isomorphic to $N\left(C\sqcup_{\nu} \left(\coprod \mathbb{I}\right)\right)$, i.e., the nerve of the pushout of $C$ along $\mathbb{I}$ at each point in $\nu$.
\end{rmk}

Now that we have defined widenings, we can define \emph{widened inclusions}.

\begin{defn}
Given an inclusion of simplicial sets $X\cin Y$ and a subset $\nu\cin Y_0$, we say that the inclusion $(\{0\}\times Y)\cup W_{\nu\cap X_0}(X)\hookrightarrow W_{\nu}(Y)$ is a \emph{widened inclusion}. More specifically, we refer to it as \emph{the inclusion $X\hookrightarrow Y$ widened at $\nu$}. We denote the class of all widened inclusions by $\Wide$.
\end{defn}

\begin{ex}\label{ex:pushoutprodiswidenedinclusion}
Given $X\cin Y$, the pushout-product
\[
(\{0\}\hookrightarrow J)\boxprod (X\hookrightarrow Y)
\]
is precisely the widened inclusion $(\{0\}\times Y)\cup W_{X_0} (X)\hookrightarrow W_{Y_0}(Y)$.
\end{ex}

\begin{ex}\label{ex:isohorninclusioniswidened}
The widened inclusion
\[
(\{0\}\times \Delta[1])\cup W_0(\partial \Delta[1])\hookrightarrow W_{0}(\Delta[1])
\]
looks like
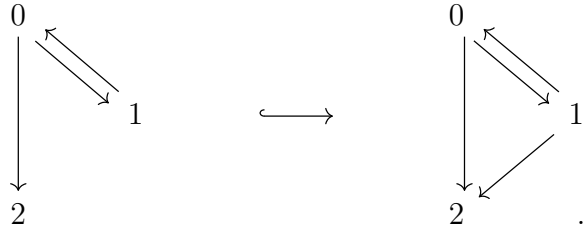
\[\begin{tikzcd}
	0 &&&& 0 \\
	& 1 & {} & {} && 1 \\
	2 &&&& 2 & {\ .}
	\arrow[from=1-1, to=3-1]
	\arrow[shift right=1, from=1-1, to=2-2]
	\arrow[hook, from=2-3, to=2-4]
	\arrow[shift right=1, from=2-2, to=1-1]
	\arrow[shift right=1, from=2-6, to=1-5]
	\arrow[shift right=1, from=1-5, to=2-6]
	\arrow[shift left=1, from=1-5, to=3-5]
	\arrow[from=2-6, to=3-5]
\end{tikzcd}\]
Similarly to Example \ref{ex:isoplexiswidening}, the widening $W_0(\Delta[1])$ is an isoplex. The simplicial set $(\{0\}\times \Delta[1])\cup W_0(\partial \Delta[1])$ is a union of all but one of the faces of this isoplex, and hence is an \emph{iso-horn}, making this widened inclusion an \emph{iso-horn inclusion}. We define these terms explictly in Section \ref{sec:isohornchar}.
\end{ex}

In order to make the two main observations of this section, we need to define a sort of partial projection map, as illustrated in the following example.

\begin{ex}\label{ex:partialproj}
There is a map $J\times \Delta[2]\to W_2(\Delta[2])$ which collapses the parts of $J\times \Delta[2]$ whose vertices are not in $J\times \{2\}$.
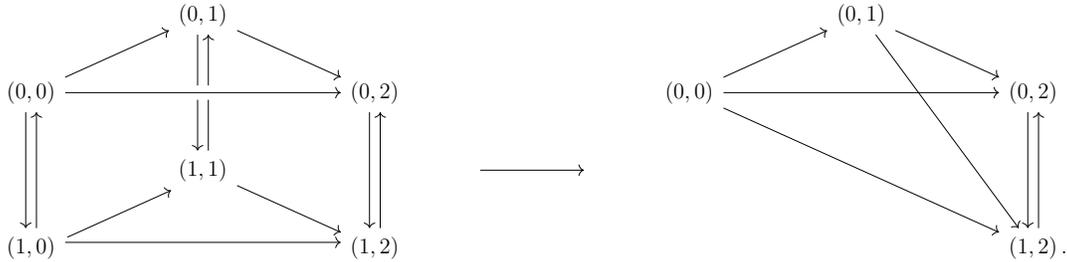
\[
\adjustbox{scale=0.7}{
\begin{tikzcd}
	&& {(0,1)} &&&&&&&& {(0,1)} \\
	{(0,0)} &&&& {(0,2)} &&&& {(0,0)} &&&& {(0,2)} \\
	&& {(1,1)} &&& {} && {} \\
	{(1,0)} &&&& {(1,2)} &&&&&&&& {(1,2)\nospace{.}}
	\arrow[from=2-9, to=1-11]
	\arrow[from=1-11, to=2-13]
	\arrow[from=1-11, to=4-13]
	\arrow[shift right=1, from=2-13, to=4-13]
	\arrow[shift right=1, from=4-13, to=2-13]
	\arrow[from=2-9, to=2-13]
	\arrow[from=2-9, to=4-13]
	\arrow[from=3-6, to=3-8]
	\arrow[shift right=1, from=4-5, to=2-5]
	\arrow[shift right=1, from=2-5, to=4-5]
	\arrow[shift left=1, from=4-1, to=4-5]
	\arrow[shift right=1, from=2-1, to=4-1]
	\arrow[from=2-1, to=1-3]
	\arrow[from=1-3, to=2-5]
	\arrow[shift right=1, from=4-1, to=3-3]
	\arrow[from=3-3, to=4-5]
	\arrow[shift right=1, from=1-3, to=3-3]
	\arrow[shift right=1, from=3-3, to=1-3]
	\arrow[shift right=1, from=4-1, to=2-1]
	\arrow[from=2-1, to=2-5, crossing over]
\end{tikzcd}
}
\]
\end{ex}

We can define partial projection maps more precisely via the following definition.

\begin{defn}
Given a simplicial set $X$ and a set of vertices $\nu\cin X_0$, we define a map $r_\nu\colon J\times X\to J$ by
\[
r_\nu ((a_0,a_1,\ldots,a_n),\sigma)=(a_0',a_1',\ldots,a_n'),
\]
where $a_i'=a_i$ if the $i$th vertex of $\sigma$ is in $\nu$ and $a_i'=0$ otherwise.
\end{defn}

Let $p_x\colon J\times X\to X$ denote the projection map. Then the image of the map
\[
(r_\nu, p_X)\colon J\times X\to J\times X
\]
is precisely $W_\nu (X)$. We denote the map $(r_\nu, p_X)$ with codomain restricted to $W_\nu (X)$ by $R_{\nu,X}$, so that we have a retract diagram
\[\begin{tikzcd}
	{W_\nu(X)} & {J\times X} & {W_\nu(X)\nospace{.}}
	\arrow[hook, from=1-1, to=1-2]
	\arrow["{R_{\nu,X}}", from=1-2, to=1-3]
\end{tikzcd}\]

The map $R_{\nu,X}$ is what we view as a sort of partial projection map, in the sense that it only collapses the parts of $X$ whose vertices are not in $\nu$.

\begin{ex}
The map in Example \ref{ex:partialproj} is precisely $R_{2,\Delta[2]}$.
\end{ex}

We can now see that widened inclusions are retracts of pushout-products.

\begin{prop}\label{prop:widenedinclusionsareretracts}
Given $X\hookrightarrow Y$ and $\nu\cin Y_0$, the widened inclusion
\[
(\{0\}\times Y)\cup W_{\nu\cap X_0} (X)\hookrightarrow W_{X_0}(Y)
\]
is a retract of the pushout-product
\[
(\{0\}\hookrightarrow J)\boxprod (X\hookrightarrow Y),
\]
and the class $\ol{\Wide}$ is equal to $\ol{(\{0\}\hookrightarrow J)\boxprod \Mono}$.
\end{prop}

\begin{proof}
Let $\nu'=\nu\cap X_0$. We can form the following retract diagram
\[\begin{tikzcd}
	{(\{0\}\times Y)\cup W_{\nu'}(X)} && {(\{0\}\times Y)\cup (J\times X)} && {(\{0\}\times Y)\cup W_{\nu'}(X)} \\
	{W_{\nu}(Y)} && {J\times Y} && {W_{\nu}(Y)\nospace{,}}
	\arrow[hook, from=1-1, to=2-1]
	\arrow[hook, from=1-5, to=2-5]
	\arrow["{R_{\nu,Y}}"', from=2-3, to=2-5]
	\arrow[hook, from=2-1, to=2-3]
	\arrow[hook, from=1-1, to=1-3]
	\arrow["{\id_{\{0\}\times Y}\cup R_{\nu',X}}", from=1-3, to=1-5]
	\arrow[hook, from=1-3, to=2-3]
\end{tikzcd}\]
proving the first statement. Since the first statement implies
\[
\ol{\Wide}\cin \ol{(\{0\}\hookrightarrow J)\boxprod \Mono},
\]
the second statement follows from observing that every map in $(\{0\}\hookrightarrow J)\boxprod \Mono$ is isomorphic to a widened inclusion as seen in Example \ref{ex:pushoutprodiswidenedinclusion}, so we have
\[
\ol{(\{0\}\hookrightarrow J)\boxprod \Mono}\cin \ol{\Wide}
\]
as well.
\end{proof}

Since $\mcS{}^{\lifts}=\ol{\mcS}{}^{\ \lifts}$ for any class of maps $\mcS$, as an immediate consequence of this proposition we get a slightly stronger characterization of the fibrant objects in the minimal model structure.

\begin{cor}
A simplicial set is fibrant in the minimal model structure if and only if it has lifts with respect to all widened inclusions.
\end{cor}

By viewing maps out of $W_{\nu}X$ as homotopies of maps out of $X$ which are fixed on the full subcomplex of $X_0\smallsetminus \nu$, we can interpret this characterization as a slightly stronger version of the homotopy extension property that we saw in Paragraph \ref{par:htpyextension}.

Taken on its own, Proposition \ref{prop:widenedinclusionsareretracts} is not especially surprising, as it follows from the known characterization of the fibrant objects in the minimal model structure. However, as we see in Section \ref{sec:isohornchar}, the concept of widened inclusions allows for a more conceptual proof of an otherwise technical result, so it is worth spelling out how the class of all widened inclusions fits into the picture.

We now turn to proving the final proposition of this section, which is a key piece of the argument in Section \ref{sec:isohornchar}. First, we need the following lemma. The idea of the lemma is that, to get a widening at some set of vertices $\mu$, we can widen first at some subset $\nu\cin \mu$ and then widen at $\mu\smallsetminus \nu$. (Technically, when taking the second widening our set of vertices is $\{0\}\times(\mu\smallsetminus\nu)$, but we slightly abuse notation and simply write $\mu\smallsetminus \nu$.)

\begin{lem}
Given a simplicial set $X$ and sets of vertices $\nu\cin \mu\cin X_0$, there is an isomorphism
\[
W_{\mu}(X)\to W_{\mu\smallsetminus \nu}(W_{\nu}(X)).
\]
\end{lem}

\begin{proof}
Let $\phi$ be the composite map
\[\begin{tikzcd}
	{J\times X} && {J\times J\times X} && {J\times W_{\nu}(X)} && {W_{\mu\smallsetminus\nu}(W_{\nu}(X))\nospace{,}}
	\arrow["{\operatorname{diag}\times \id_X}", from=1-1, to=1-3]
	\arrow["{\id_{J}\times R_{\nu,X}}", from=1-3, to=1-5]
	\arrow["{R_{\mu\smallsetminus\nu,W_{\nu}(X)}}", from=1-5, to=1-7]
\end{tikzcd}\]
so that
\[
\phi(a,\sigma)=(r_{\mu\smallsetminus\nu}(a,\sigma),(r_{\nu}(a,\sigma),\sigma)).
\]
Let $\phi'$ denote the restriction of $\phi$ to $W_{\mu}(X)$. We want to show that $\phi'$ is bijective.

To show injectivity of $\phi'$, take $\alpha=((a_0,a_1,\ldots,a_n),\sigma)$ and $\beta=((b_0,b_1,\ldots,b_n),\sigma')$ such that $\phi'(\alpha)=\phi'(\beta)$. Since the third coordinate is given by projection onto $X$, it is immediate that $\sigma=\sigma'$, so it only remains to show that $a_i=b_i$ for each $0\leq i\leq n$. If the $i$th vertex of $\sigma$ is not in $\mu$, then $a_i=b_i=0$ because $\alpha$ and $\beta$ are in $W_{\mu}(X)$. If the $i$th vertex of $\sigma$ is in $\mu\smallsetminus \nu$, then the $i$th coordinate of $r_{\mu\smallsetminus\nu}(\alpha)$ is $a_i$ and the $i$th coordinate of $r_{\mu\smallsetminus\nu}(\beta)$ is $b_i$, so they must be equal. A similar argument applies for $r_{\nu}$ in the remaining case that the $i$th vertex of $\sigma$ is in $\nu$.

To show surjectivity of $\phi'$, take $\gamma=((b_0,\ldots,b_n),((c_0,\ldots,c_n),\sigma))$ in $W_{\mu\smallsetminus\nu}(W_{\nu}(X))$. We define $(a_0,\ldots,a_n)$ by letting $a_i=b_i$ if the $i$th vertex of $\sigma$ is in $\mu\smallsetminus\nu$, letting $a_i=c_i$ if the $i$th vertex of $\sigma$ is in $\nu$, and letting $a_i=0$ if the $i$th vertex of $\sigma$ is not in $\mu$. Then $((a_0,\ldots,a_n),\sigma)$ is in $W_{\mu}(X)$ and is sent to $\gamma$ by $\phi$.
\end{proof}

Using this lemma, we get the following proposition, which says that widened inclusions are built out of simpler widened inclusions.

\begin{prop}\label{prop:factoringwidenedinclusions}
Given an inclusion $X\hookrightarrow Y$ and sets of vertices $\nu\cin \mu\cin Y_0$, the widened inclusion $Y\cup W_{\mu\cap X_0}(X)\hookrightarrow W_\mu(Y)$ is the composite of a pushout of an inclusion widened at $\nu$ followed by a map isomorphic to an inclusion widened at $\mu\smallsetminus\nu$.
\end{prop}

\begin{proof}
Let $\nu'=\nu\cap X_0$ and $\mu'=\mu\cap X_0$. The diagram
\[
\adjustbox{scale=0.8}{
\begin{tikzcd}
	{Y\cup W_{\mu'}(X)} && {W_\mu(Y)} & {W_{\mu\smallsetminus\nu}(W_\nu(Y))} \\
	{(Y\cup W_{\nu}(X))\cup W_\mu'(X)} \\
	& {W_{\nu}(Y)\cup W_{\mu'}(X)} & {W_{\nu}(Y)\cup W_{\mu'\smallsetminus \nu'}(W_{\nu'}(X))}
	\arrow[equal, from=1-1, to=2-1]
	\arrow["\cong"', from=3-2, to=3-3]
	\arrow["\cong", from=1-3, to=1-4]
	\arrow["{(a)}"', hook, from=2-1, to=3-2]
	\arrow[hook, from=3-3, to=1-4]
	\arrow[hook, from=1-1, to=1-3]
	\arrow["{(b)}"', hook, from=3-2, to=1-3]
\end{tikzcd}
}
\]
shows how $Y\cup W_{\mu'}(X)\hookrightarrow W_{\mu}(Y)$ factors as the map $(a)$, which is a pushout of the inclusion $X\hookrightarrow Y$ widened at $\nu$, followed by the map $(b)$, which is isomorphic to the inclusion $W_{\nu'}(X)\hookrightarrow W_{\nu}(Y)$ widened at $\mu\smallsetminus\nu$.
\end{proof}

Let $\Wide^{(1)}$ denote the class of inclusions widened at a single vertex $\nu=\{v\}$. By iterated application of Proposition \ref{prop:factoringwidenedinclusions}, we can build any widened inclusion out of maps in $\Wide^{(1)}$, hence the following corollary.

\begin{cor}\label{cor:factoringwidenedinclusions}
The class $\Wide$ is contained in $\ol{\Wide^{(1)}}$.
\end{cor}

Now that we have some basic facts about widened inclusions, we are ready to apply them to prove our main result in the next section.

\section{A characterization of the fibrant objects in terms of iso-horn inclusions}\label{sec:isohornchar}

In this section, we define iso-horn inclusions and show that they generate the same saturated class as $(\{0\}\hookrightarrow J)\boxprod \Bdry$, hence showing that the fibrant objects in the minimal model structure are precisely the simplicial sets with iso-horn extensions. As mentioned already in Example \ref{ex:isohorninclusioniswidened}, each iso-horn inclusion is an example of a widened inclusion. We therefore have the arrow $(a)$ in the diagram
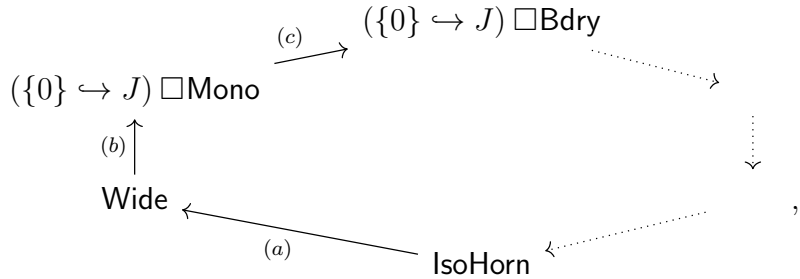
\[\begin{tikzcd}[row sep=tiny]
	& {\left(\{0\}\hookrightarrow J\right)\boxprod \Bdry} \\
	{\left(\{0\}\hookrightarrow J\right)\boxprod \Mono} && {\hphantom{\Wide_{\operatorname{nar}}}\vphantom{{}_{\operatorname{nar}}}} \\
	{\vspace{2mm}}\\
	{\Wide} && {\hphantom{\Wide}\vphantom{{}^{(1)}}\nospace{,}}\\
	& {\IsoHorn}
	\arrow["{(a)}", from=5-2, to=4-1]
	\arrow["{(b)}", from=4-1, to=2-1]
	\arrow["{(c)}", from=2-1, to=1-2]
	\arrow[dotted, from=1-2, to=2-3]
	\arrow[dotted, from=2-3, to=4-3]
	\arrow[dotted, from=4-3, to=5-2]
\end{tikzcd}\]
where an arrow $\mcS\to \mcT$ indicates that $\mcS\cin \ol{\mcT}$. We get $(b)$ from Proposition \ref{prop:widenedinclusionsareretracts} and $(c)$ from Lemma \ref{lem:equivofan1givespushoutprod}. The goal of this section is to fill out the dotted arrows in this diagram to create a loop, implying that all of these classes generate the same saturated class.

We begin by introducing the concept of \emph{narrow vertices} of a simplicial set. At the end of this section we see that restricting to the special case of inclusions widened at sets of narrow vertices is helpful in the proof of Proposition \ref{prop:isohornskeletalfiltration}.

\begin{defn}
Given a simplicial set $X$, a vertex $v\in X_0$ is \emph{narrow} if, for every non-degenerate simplex $\sigma$ in $X$, at most one vertex of $\sigma$ is $v$. A set of vertices $\nu\cin X_0$ is \emph{narrow} if every vertex in $\nu$ is narrow. Denote by $\Wide_{\operatorname{nar}}$ the class of inclusions $X\hookrightarrow Y$ widened at a set of vertices which are narrow in $Y$.
\end{defn}

\begin{ex}
Each vertex of $J$ is not narrow because there exist non-degenerate 2-simplices $0\to 1\to 0$ and $1\to 0\to 1$.
\end{ex}

\begin{ex}\label{ex:pushoutproductsarenarrow}
For all $n\geq 0$, every vertex of $\Delta[n]$ is narrow. For $n\geq 1$, we saw in Example \ref{ex:pushoutprodiswidenedinclusion} that the pushout-product
\[
(\{0\}\hookrightarrow J)\boxprod (\partial\Delta[n]\hookrightarrow \Delta[n])
\]
is the widened inclusion
\[
\Delta[n]\cup W_{\Delta[n]_0}(\partial\Delta[n])\hookrightarrow W_{\Delta[n]_0}(\Delta[n]),
\]
hence
\[
(\{0\}\hookrightarrow J)\boxprod \Bdry \cin \Wide_{\operatorname{nar}}.
\]
\end{ex}

\begin{ex}
If $x$ is a narrow vertex of some simplicial set $Y$, then, for every $X\cin Y$ containing $x$, the vertex $x$ is also narrow in $X$.
\end{ex}

Now that we have defined narrow vertices, we note that in the factoring of a widened inclusion as in Proposition \ref{prop:factoringwidenedinclusions}, if our original set of vertices $\mu$ at which we are widening is narrow, then so are the subsets $\nu$ and $\mu\smallsetminus \nu$. (More precisely, if $\mu$ is narrow in $Y$ then $\nu$ is narrow in $Y$ and $\mu\smallsetminus \nu$ is narrow in $W_{\nu}(Y)$.) Thus, we get the following ``narrow'' version of Corollary \ref{cor:factoringwidenedinclusions}, where we denote by $\Wide_{\operatorname{nar}}^{(1)}$ the class of inclusions widened at a single narrow vertex.

\begin{prop}\label{prop:factoringnarrow}
The class $\Wide_{\operatorname{nar}}$ is contained in $\ol{\Wide_{\operatorname{nar}}^{(1)}}$.
\end{prop}

We define iso-horn inclusions to be particular maps in $\Wide_{\operatorname{nar}}^{(1)}$.

\begin{defn}\label{def:isohorns}
Fix $n\geq 1$ and $0\leq i\leq n-1$. We let $\tri_i[n]=W_i(\Delta[n-1])$, which we call an \emph{isoplex}, and let $\bV_i[n]=W_i(\partial\Delta[n-1])\cup \Delta[n-1]$ and call it an \emph{iso-horn}. We call the inclusion $\bV_i[n]\hookrightarrow \tri_i[n]$ an \emph{iso-horn inclusion}. Denote by $\IsoHorn$ the set of all iso-horn inclusions.
\end{defn}

Recall that the standard $n$-simplex $\Delta[n]$ is the nerve of the category $[n]$. We have chosen the terminology ``isoplex'' because $\tri_i[n]$ is the nerve of the category we denote by $[n]_i$, which is $[n]$ with the edge $i\to i+1$ inverted. Continuing this analogy, we can view a face of $\tri_i[n]$ as the result of deleting a single vertex, in which case the iso-horn $\bV_i[n]$ is the union of all of the faces of $\tri_i[n]$ except for the $i$th face. Let us make this terminology more precise.

\begin{defn}
For $0\leq j\leq n$, viewing $\tri_i[n]$ as the nerve of the category $[n]_i$
\[\begin{tikzcd}[column sep=small]
	0 & \ldots & {i-1} & i & {i+1} & {i+2} & \ldots & n\nospace{,}
	\arrow[from=1-1, to=1-2]
	\arrow[from=1-2, to=1-3]
	\arrow[from=1-3, to=1-4]
	\arrow[shift right=1, from=1-4, to=1-5]
	\arrow[shift right=1, from=1-5, to=1-4]
	\arrow[from=1-5, to=1-6]
	\arrow[from=1-6, to=1-7]
	\arrow[from=1-7, to=1-8]
\end{tikzcd}\]
we let the $j$th face of the isoplex $\tri_i[n]$, denoted $d_j \tri_i[n]$, be the full subcomplex on all but the $j$th vertex.
\end{defn}

Observe that when $j=i$ or $i+1$, the $j$th face of $\tri_i[n]$ is an $(n-1)$-simplex $\Delta[n-1]$, and otherwise it is an $(n-1)$-isoplex. We can view an isoplex as an ``isomorphism  of $(n-1)$-simplices'' from the $i+1$ face to the $i$th face. Similarly, we can view an iso-horn $\bV_i[n]$ as an $(n-1)$-simplex extended by an isomorphism along its boundary.

We have already seen examples of isoplexes and iso-horn inclusions in Examples \ref{ex:isoplexiswidening} and \ref{ex:isohorninclusioniswidened}.

We have now defined the classes necessary to fill out our loop from earlier:
\[\begin{tikzcd}[row sep=tiny]
	& {\left(\{0\}\hookrightarrow J\right)\boxprod \Bdry} \\
	{\left(\{0\}\hookrightarrow J\right)\boxprod \Mono} && {\Wide_{\operatorname{nar}}} \\
	{\vspace{2mm}}\\
	{\Wide} && {\Wide_{\operatorname{nar}}^{(1)}\nospace{,}} \\
	& {\IsoHorn}
	\arrow["{(d)}", from=1-2, to=2-3]
	\arrow["{(e)}", from=2-3, to=4-3]
	\arrow["{(f)}", from=4-3, to=5-2]
	\arrow["{(a)}", from=5-2, to=4-1]
	\arrow["{(b)}", from=4-1, to=2-1]
	\arrow["{(c)}", from=2-1, to=1-2]
\end{tikzcd}\]
where again an arrow $\mcS\to \mcT$ indicates that $\mcS\cin \ol{\mcT}$. We get $(d)$ from Example \ref{ex:pushoutproductsarenarrow}, and $(e)$ from Proposition \ref{prop:factoringnarrow}, so it only remains to show $(f)$ in order to conclude that these classes all generate the same saturated class.

The iso-horn inclusions are the simplest possible examples of widened inclusions, since each iso-horn inclusion is a boundary inclusion widened at just one vertex. From this perspective, they are the fundamental building blocks of the class of widened inclusions, similar to how the boundary inclusions are the fundamental building blocks of the monomorphisms. In fact, our approach to proving $(f)$ is to adapt the standard proof that $\ol{\Bdry}=\Mono$.

The idea in the proposition below is that, for some inclusion $X\hookrightarrow Y$ and a narrow $y\in Y$, to build up the widened inclusion $Y\cup W_y(X)\hookrightarrow W_y(Y)$ we inductively widen the $k$-simplices of $Y$ which are not in $X$. That is, having widened the $(k-1)$-skeleton, for any non-degenerate $k$-simplex $\sigma$ in $Y\smallsetminus X$ with $y$ as its $i$th vertex, the boundary of $\sigma$ has already been widened, and so we have a map from the iso-horn $\tri_i[k+1]$ which sends the $i+1$ face to $\sigma$ and sends the other faces to the widening of the boundary of $\sigma$. We therefore widen $\sigma$ by taking the pushout of the iso-horn inclusion $\bV_i[k+1]\hookrightarrow \tri_i[k+1]$. We spell out this process more explicitly in the proof, for which we need the notation from the following definition.

\begin{defn}
Fix a simplicial set $Y$ and a narrow vertex $y\in Y_0$. Given $k\geq 1$ and a non-degenerate $k$-simplex $\sigma\colon \Delta[k]\to Y$ with $y$ as a vertex, let $0\leq i\leq k$ be the index such that $y$ is the $i$th vertex of $\sigma$. Then we let $W_y(\sigma)$ and $W_y(\partial\sigma)$ be the restrictions of the map $\id_J\times \sigma$ as in
\[\begin{tikzcd}
	{W_i(\partial\Delta[k])} & {W_y(\sk_{k-1}Y)} \\
	{\tri_i[k+1]} & {W_y(\sk_k Y)} \\
	{J\times\Delta[k]} & {J\times\sk_k Y\nospace{.}}
	\arrow["{W_y(\sigma)}", from=2-1, to=2-2]
	\arrow[hook, from=1-1, to=2-1]
	\arrow["{W_y(\partial\sigma)}", from=1-1, to=1-2]
	\arrow[hook, from=1-2, to=2-2]
	\arrow["{\id_J\times \sigma}"', from=3-1, to=3-2]
	\arrow[hook, from=2-2, to=3-2]
	\arrow[hook, from=2-1, to=3-1]
\end{tikzcd}\]
(Recall that $\sk_k X$ denotes the $k$-skeleton of $X$.)
\end{defn}

\begin{prop}\label{prop:isohornskeletalfiltration}
The class $\Wide^{(1)}_{\operatorname{nar}}$ is contained in $\ol{\IsoHorn}$.
\end{prop}

\begin{proof}
Fix an inclusion $X\hookrightarrow Y$ and a narrow vertex $y\in Y_0$. We would like to show that the inclusion $X\hookrightarrow Y$ widened at $y$ is a countable composition of pushouts of coproducts of iso-horn inclusions. There are two cases: either $y$ is in $X$ or it is not. We first argue that the case that $y$ is not in $X$ reduces to the other case.

If $y$ is not in $X$, then $X\hookrightarrow Y$ widened at $y$ is $Y\hookrightarrow W_y(Y)$, which we can factor as
\[\begin{tikzcd}
	Y & {Y\cup W_y(X\cup \{y\})} & {W_y(Y)\nospace{,}}
	\arrow[hook, from=1-1, to=1-2]
	\arrow[hook, from=1-2, to=1-3]
\end{tikzcd}\]
where the first inclusion is the same as $Y\hookrightarrow Y\cup W_y(\{y\})$ since $\{y\}$ is disjoint from $X$, and so is a pushout of $\{0\}\hookrightarrow J$ which is the iso-horn inclusion $\bV_0[1] \hookrightarrow \tri_0[1]$. Since the remaining inclusion is $(X\cup \{y\})\hookrightarrow Y$ widened at $y$, we have reduced to the other case.

If $y$ is in $X$, then $W_y(\sk_0 Y)\cin Y\cup W_y(X)$, so we may factor the widened inclusion $Y\cup W_y(X)\hookrightarrow W_y (Y)$ as
\[
\adjustbox{scale=1}{
\begin{tikzcd}[column sep=small]
	{Y\cup W_y(X)\cup W_y(\sk_0 Y)} & \ldots & {Y\cup W_y(X)\cup W_y(\sk_k Y)} & \ldots & {W_y(Y)\nospace{,}}
	\arrow[hook, from=1-1, to=1-2]
	\arrow[hook, from=1-2, to=1-3]
	\arrow[hook, from=1-3, to=1-4]
	\arrow[hook, from=1-4, to=1-5]
\end{tikzcd}
}
\]
so it suffices to show that the inclusion
\[\begin{tikzcd}[column sep=small]
	{Y\cup W_y(X)\cup W_y(\sk_{k-1} Y)} && {Y\cup W_y(X)\cup W_y(\sk_{k} Y)}
	\arrow[hook, from=1-1, to=1-3]
\end{tikzcd}\]
is a pushout of a coproduct of iso-horn inclusions for each $k\geq 1$, as in the diagram
\[\begin{tikzcd}
	{\coprod\limits_{\sigma\in\mcN_{k}} \bV_{i_\sigma}[k+1]} &&& {Y\cup W_y(X)\cup W_y(\sk_{k-1} Y)} \\
	\\
	{\coprod\limits_{\sigma\in\mcN_{k}}\tri_{i_\sigma}[k+1]} &&& {Y\cup W_y(X)\cup W_y(\sk_{k} Y)\nospace{,}} \\
	&& {}
	\arrow[hook, from=1-4, to=3-4]
	\arrow[hook, from=1-1, to=3-1]
	\arrow["{\coprod\limits_{\sigma\in\mcN_{k}} \sigma\cup W_y(\partial \sigma)}", from=1-1, to=1-4]
	\arrow["{\coprod\limits_{\sigma\in\mcN_{k}} W_y(\sigma)}"', from=3-1, to=3-4]
\end{tikzcd}\]
where $\mcN_{k}$ is the set of non-degenerate $k$-simplices of $Y$ which have $y$ as a vertex and are not in $X$, and where $i_\sigma$ is the index $0\leq i_\sigma\leq k$ such that $y$ is the $i_\sigma$th vertex of $\sigma$. 
\end{proof}

Having proved that $\Wide$ and $\IsoHorn$ generate the same saturated class as $\mcA=(\{0\}\hookrightarrow J)\boxprod \Bdry$, our main theorem is an immediate corollary.

\begin{thm}\label{thm:isohornminimalfibrant}
Given a simplicial set $X$, the following are equivalent.
\begin{enumerate}
    \item The simplicial set $X$ is fibrant in the minimal model structure.
    \item The map $X\to \ast$ has the right lifting property with respect to the class of widened inclusions.
    \item The map $X\to \ast$ has the right lifting property with respect to the set of iso-horn inclusions.
\end{enumerate}
\end{thm}

We conclude with a remark considering how our work might be extended to presheaf categories on categories other than $\Delta$, such as the tree category $\Omega$ and Joyal's cell category $\Theta$.

\begin{rmk}\label{rmk:othercategories}
Given a small category $A$, Cisinski's theory implies that $\Set^{A^{\op}}$, the category of presheaves on $A$, has a minimal model structure. In fact, given any trivially fibrant presheaf $L$ with two distinct maps from the terminal object $0,1\colon \ast \to L$, one can trace through a similar argument as in Section \ref{sec:knownchar} to get a characterization of the fibrant objects of the minimal model structure on $\Set^{A^{\op}}$ in terms of pushout-products involving the maps $0,1\colon \ast \to L$. The difficulty in extending our results for $\sSet$ in $\Set^{A^{\op}}$ lies in finding an $L$ which is nice enough to replicate the methods in Sections \ref{sec:wide} and \ref{sec:isohornchar}. In particular, it should allow for a good definition of widenings and partial projection maps. If one can axiomatize the essential properties of $J$ which underpin our methods, then our proof could be replicated for any presheaf category with a presheaf $L$ satisfying those axioms. The remaining question would then be whether such a nice presheaf $L$ necessarily exists for every small category $A$, or if there are reasonable sufficient conditions on $A$ for such an $L$ to exist which are satisfied by important categories like $\Omega$ and $\Theta$.
\end{rmk}

\bibliographystyle{plain}
\bibliography{Minimal}

\end{document}